\documentclass[11pt]{amsart}
\usepackage{graphicx}
\usepackage{amssymb}
\usepackage{ytableau}
\usepackage{mathtools,color}
\usepackage{amsrefs}
\usepackage{leftidx}
\usepackage{subcaption}
\usepackage{tikz}
\usepackage{tikz-3dplot}
\usepackage{eucal}
\usepackage{float}
\usepackage{fullpage}
\title{The sum of all width-one tensors}
\author{William Q. Erickson}
\address{
William Q.~Erickson\\
Department of Mathematics\\
Baylor University \\ 
One Bear Place \#97328\\
Waco, TX 76798} 
\email{Will\_Erickson@baylor.edu}

\author{Jan Kretschmann}
\address{
Jan Kretschmann\\
Department of Mathematical Sciences\\
University of Wisconsin--Milwaukee \\ 
3200 N.~Cramer St.\\
Milwaukee, WI 53211} 
\email{kretsc23@uwm.edu}

\newcommand{\T}{\mathcal{T}}
\renewcommand{\emptyset}{\text{\O}}

\newcommand{\la}{\lambda}

\newcommand{\x}{\mathbf{x}}
\newcommand{\n}{\mathbf{n}}

\theoremstyle{plain}
\newtheorem{theorem}{Theorem}[section]

\newtheorem{lemma}[theorem]{Lemma}

\theoremstyle{definition}
\newtheorem{definition}[theorem]{Definition}
\newtheorem{rem}[theorem]{Remark}



\begin{document}

\begin{abstract}
    This paper generalizes a recent result by the authors concerning the sum of width-one matrices; in the present work, we consider width-one tensors of arbitrary dimensions.
    A tensor is said to have width 1 if, when visualized as an array, its nonzero entries lie along a path consisting of steps in the directions of the standard coordinate vectors. 
    We prove two different formulas to compute the sum of all width-one tensors with fixed dimensions and fixed sum of (nonnegative integer) components.
    The first formula is obtained by converting width-one tensors into tuples of one-row semistandard Young tableaux;
    the second formula, which extracts coefficients from products of multiset Eulerian polynomials, is derived via Stanley--Reisner theory, making use of the EL-shelling of the order complex on the standard basis of tensors.
    
\end{abstract}

\maketitle

\section{Introduction}
Let $\T^1_{\n,s}$ be the set of all $\n=(n_1,n_2,\ldots,n_d)$-dimensional tensors (equivalently, hypermatrices) with nonnegative integer 
entries summing to $s$, such that the nonzero entries lie on a single path consisting of steps in the directions of the standard basis vectors $e_1, \ldots, e_d$.
For example, Figure \ref{fig:T-1-333-example} shows a typical element of $\T^1_{(3,3,3),s}$, where the nonzero entries sum to $s$, and all lie on the lattice points along the marked path.
\begin{figure}[H]
    \centering
    \tdplotsetmaincoords{70}{110} % 

\tdplotsetmaincoords{70}{120} % set viewpoint 
\tdplotsetrotatedcoords{0}{0}{0} %<- rotate around (z,y,z)
\begin{tikzpicture}[scale=1.5,tdplot_rotated_coords,
            blackBall/.style={ball color = black!80},
                    borderBall/.style={ball color = white,opacity=.25}, 
                    very thick, lightgray]

\foreach \x in {0,1,2}
   \foreach \y in {0,1,2}
      \foreach \z in {0,1,2}{
           %#####################################################
           \ifthenelse{  \lengthtest{\x pt < 2pt}  }{
             \draw (\x,\y,\z) -- (\x+1,\y,\z);
             \shade[blackBall] (\x,\y,\z) circle (0.025cm); 
           }{}
           %#####################################################
           \ifthenelse{  \lengthtest{\y pt < 2pt}  }{
               \draw (\x,\y,\z) -- (\x,\y+1,\z);
               \shade[blackBall] (\x,\y,\z) circle (0.025cm);
           }{}
           %#####################################################
           \ifthenelse{  \lengthtest{\z pt < 2pt}  }{
               \draw (\x,\y,\z) -- (\x,\y,\z+1);
               }{}
}

\draw[ultra thick,black] (2,0,2) -- (1,0,2) -- (1,1,2) -- (1,1,1) -- (1,2,1) -- (1,2,0) -- (0,2,0);
\draw[very thick, lightgray] (2,1,2) -- (2,2,2) -- (2,2,1);

% \shade[blackBall] (1,0,2) circle (0.08cm); 
% \shade[blackBall] (1,1,2) circle (0.08cm);
% \shade[blackBall] (1,2,1) circle (0.08cm);

% \shade[borderBall] (1,0,2) circle (0.15cm);
% \shade[borderBall] (1,1,2) circle (0.15cm);
% \shade[borderBall] (1,2,1) circle (0.15cm);

\shade[blackBall] (2,2,2) circle (0.025cm);
\shade[blackBall] (2,2,1) circle (0.025cm);
\shade[blackBall] (2,2,0) circle (0.025cm);

\fill[black] 
(2,0,2) circle (0pt) node[left] {$(1,1,1)$}
(0,2,0) circle (0pt) node[right] {$\n = (3,3,3)$};

\end{tikzpicture}
    \caption{Visualization of an element of $\T^1_{(3,3,3),s}$.}
    \label{fig:T-1-333-example}
\end{figure}
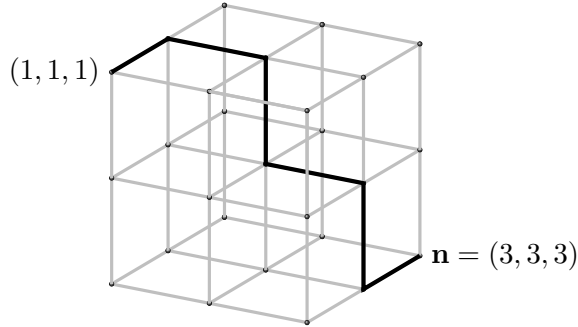

Computing the sum of such matrices in $d=2$ dimensions has useful applications in optimal transport~\cite{EK23}, drastically simplifying the problem of computing the expected value of the one-dimensional earth mover's distance (EMD) between two compositions, as first done by Bourn and Willenbring \cite{BW20}.
It is natural to think about a generalization of the EMD between an arbitrary number $d$ of compositions; this problem is addressed, for instance, in~\cite{Bein} and~\cite{Kline}.
The recursive methods of~\cite{BW20} were generalized in~\cite{E21} to find a recursion for the expected value of the $d$-fold EMD.
Hence this paper plays an analogous role, in generalizing the explicit formulas for two-dimensional matrices in~\cite{EK23} to $d$-dimensional tensors.
In particular, the main results in this paper are two explicit formulas for the sum $\Sigma^1_{\n,s}$ of all tensors in $\T^1_{\n,s}$.
The first is shown in Section~\ref{sec:RSK}, which follows from a translation of the problem into the language Young Tableaux. 
We set up a bijection between any tensor $T \in \T^1_{\n,s}$ and a list of one-row semistandard tableaux, which allows us to write down a formula for our desired sum in terms of binomial coefficients (Theorem~\ref{thm:main result tableaux}).

In Sections~\ref{sec:ASC} and~\ref{sec:Eulerian polys}, we approach the question through the lens of Stanley--Reisner theory. 
We describe the order complex on the standard basis of the $\n$-dimensional tensors, and we use a special case of an EL-shelling to find the corresponding $h$-polynomial. 
This $h$-polynomial turns out to be a multiset Eulerian polynomial, as we show in Section~\ref{sec:Eulerian polys}; these polynomials, studied by MacMahon and many others since, enumerate the descents in multiset permutations.
We conclude by presenting a second explicit formula for $\Sigma^1_{\n,s}$ using techniques from Stanley--Reisner theory (Theorem~\ref{thm:main result h polys}).

Similarly to our two formulas for matrices in \cite{EK23}, the two formulas in this paper are opposite each other with regard to computing time.
Although this issue falls outside the focus of the paper, nonetheless it is not hard to verify that Theorem \ref{thm:main result tableaux} outperforms Theorem \ref{thm:main result h polys} as $s$ increases for fixed $\n$; the opposite is true, however, as the dimension $d$ or the parameters $n_i$ increase for fixed $s$.
% In \cite{EK23}, the sum of width-1 contingency tables, so $n \times n$ matrices, was computed by means of a shelling of the order complex $[n]\times[n]$.
% To extend this result, we want to increase the number of dimensions to arbitrary size.
% Our goal is to compute the sum of all $n\times n \times \ldots \times n$ tensors.
% Our goal of this paper is to compute $\SS(d;{n_1\times n_2\times \cdots \times n_k})$, a the sum of all $k$-dimensional width-1 tensors with grand total $d$.
% \jk{Do we need to give an intro to complexes again?}

\section{Notation and statement of the problem}

Throughout the paper, a boldface letter denotes an element of $(\mathbb{Z}_{>0})^d$; in particular, $\n \coloneqq (n_1, \ldots, n_d)$ and $\x \coloneqq (x_1, \ldots, x_d)$.
We write $|\x| \coloneqq \sum_i x_i$, as well as $\min(\x) \coloneqq \min_i\{x_i\}$ and $\max(\x) \coloneqq \max_i\{x_i\}$.
Let $\mathbf{1} \coloneqq (1, \ldots, 1)$.
As usual, we write $[x] \coloneqq \{1, \ldots, x\}$.
Define the poset
\begin{equation}
    \label{def:Pi}
    \Pi_\x \coloneqq [x_1] \times \cdots \times [x_d] = \{(a_1, \ldots, a_d) : 1 \leq a_i \leq x_i \text{ for all $i$}\},
\end{equation}
equipped with the product order, so that $\mathbf a \leq \mathbf b \iff a_i \leq b_i$ for each $i= 1, \ldots, d$.
Because our main problem addresses $n_1 \times \cdots \times n_d$ tensors, we will always be working inside $\Pi_{\n}$, but it will be useful to consider the subposets $\Pi_{\x}$, which are the lower-order ideals generated by each $\x \in \Pi_{\n}$.

Recall that a \emph{chain} is a totally ordered subset of $\Pi_{\x}$, and an \emph{antichain} is a subset whose elements are pairwise incomparable. 
The \emph{width} of a subset $S \subseteq \Pi_{\x}$ is the size of the largest antichain contained in $S$. 
Equivalently, by Dilworth's theorem, the width of $S$ equals the minimum number of chains into which $S$ can be partitioned.
In particular, $S$ has width 1 if and only if $S$ is a chain.

In this paper, we consider certain tensors of order $d$.
Equivalently, the reader may prefer to consider $d$-dimensional arrays (also called hypermatrices).
Taking the real numbers $\mathbb{R}$ as our ground field, we let 
\[
\T_\n \coloneqq \mathbb{R}^{n_1} \otimes \cdots \otimes \mathbb{R}^{n_d}
\]
denote the space of order-$d$ tensors with dimensions $\n$.
Upon fixing the standard basis $\{e_1, \ldots e_{n_i}\}$ for each factor $\mathbb{R}^{n_i}$, every tensor $T \in \T_\n$ can be written uniquely in the form
\begin{equation} 
    \label{tensor form}
    T = \sum_{\x \in \Pi_{\n}} T[\x] \: e_{x_1} \otimes \cdots \otimes e_{x_d},
\end{equation}
where the scalars $T[\x] \in \mathbb{R}$ are called the \emph{components} of $T$.
Hence $T$ is described completely by its components $T[\x]$, which can be regarded as the entries in an $n_1 \times \cdots \times n_d$ array.  
The \emph{elementary tensor} $E_{\x}$ is given by 
\begin{equation}
    \label{Ex}
    E_{\x}[\mathbf{y}] = \delta_{\x\mathbf{y}},
\end{equation}
where $\delta$ is the Kronecker delta.
Hence $E_\x$ can be regarded as an array with 1 in position $\x$ and 0's elsewhere.

The \emph{support} of a tensor $T$ is the set
\[
\operatorname{supp}(T) \coloneqq \{\x \in \Pi_{\n} : T[\x] \neq 0\}.
\]
We say that $T$ is a \emph{width-one} tensor if $\operatorname{supp}(T)$ has width 1 as a subposet of $\Pi_{\n}$. 
In this paper, we restrict our attention to those width-one tensors whose components are nonnegative integers summing to some positive integer $s$.
We denote this set by
\begin{equation} 
    \label{T1ns def}
    \T^{1}_{\n,s} \coloneqq \left\{T \in \T_\n : \quad 
    \parbox{.28\linewidth}{
    $T$ has width 1,\\
    $T[\x] \in \mathbb{Z}_{\geq 0}$ for all $\x \in \Pi_\n$,\\
    $\sum_\x T[\x] = s$.
    }
\right\}.
\end{equation}
The main problem of this paper is to write down an explicit formula for the sum of all tensors in $\T^{1}_{\n,s}$, under the usual component-wise addition.
We denote this tensor by
\begin{equation}
    \label{Sigma1ns def}
    \Sigma^{1}_{\n,s} \coloneqq \sum_{\mathclap{T \in \T^{1}_{\n,s}}} T.
\end{equation}

\section{Main result, first version}
\label{sec:RSK}

We generalize the argument in~\cite{EK23}*{\S3}.
It follows from~\eqref{tensor form} and~\eqref{Ex} that each tensor $T \in \T_\n$ can be written as a unique sum of elementary tensors
\[
T = \sum_{\x \in \Pi_{\n}} T[\x] E_\x = \sum_{\x \in {\rm supp}(T)} T[\x] E_\x.
\]
Now suppose that $T \in \T^{1}_{\n,s}$, as defined in~\eqref{T1ns def}.  
Then by definition, the $\x$'s appearing in the right-hand sum above form a chain in $\Pi_\n$.
Writing out all of these $\x$'s as column vectors in ascending order, with multiplicities $T[\x]$, we obtain a $d \times s$ matrix.
Note that the entries in each row of this matrix are weakly increasing, and the entries in the $i$th row are elements in $[n_i]$.
We will write each row as a row of boxes.
(We do this primarily to evoke a one-row semistandard Young tableau, which is the same thing as a weakly increasing sequence.)
We thus obtain $d$ rows with $s$ boxes each:
\[
\ytableausetup{boxsize=1.3em, centertableaux}
\begin{array}{cccc}
\ydiagram{3}& \cdots \quad \ytableaushort{{x_1}} \quad \cdots & \ydiagram{3} & \quad \leftarrow \text{weakly increasing in $[n_1]$}\\
\ydiagram{3} & \cdots \quad \ytableaushort{{x_2}} \quad \cdots & \ydiagram{3} & \quad \leftarrow \text{weakly increasing in $[n_2]$}\\
\vdots & \vdots & \vdots  & \vdots\\
\ydiagram{3}& \cdots \quad \ytableaushort{{x_d}} \quad \cdots & \ydiagram{3} & \quad \leftarrow \text{weakly increasing in $[n_d]$}
\end{array}
\]
This procedure is clearly invertible: given $d$ rows of length $s$, with entries in the $i$th row weakly increasing in $[n_i]$, we recover the associated width-one tensor in $\T^{1}_{\n,s}$ by summing the $s$ elementary tensors $E_\x$, for each column $\x$.
In our picture above, we have filled in the entries $\x = (x_1, \ldots, x_d)$ for one typical column; note that this column contributes 1 to the component $T[\x]$.
It follows that we have a bijection between $\T^{1}_{\n,s}$ and the set of $d$-tuples of weakly increasing sequences in $[n_1], \ldots, [n_d]$.

It will be a useful fact that the number of weakly increasing sequences of length $\ell$, taken from the set $[p]$, equals the number of (weak) integer compositions of $\ell$ into $p$ parts, which is well known to be
\begin{equation}
\label{composition formula}
    \binom{\ell + p - 1}{\ell} = \binom{\ell + p - 1}{p-1}.
\end{equation}

\begin{theorem}
    \label{thm:main result tableaux}
    Let $\Sigma^1_{\n,s}$ be the sum of all tensors in $\T^{1}_{\n,s}$.
    For each $\x \in \Pi_\n$, we have
    \[
    \Sigma^{1}_{\n,s}[\x] = \sum_{j=1}^s \prod_{i=1}^d \binom{x_i + j - 2}{j-1} \binom{n_i - x_i + s - j}{s -j}.
    \]
\end{theorem}

\begin{proof}
    For each $T \in \T^{1}_{\n,s}$, consider its corresponding $d$-tuple of weakly increasing rows, as described above.
    Recall that each column $\x$ contributes 1 to the component $T[\x]$.
    Hence the component $\Sigma^{1}_{\n,s}[\x]$ equals the number of occurrences of the column $\x$, counted in all possible $d$-tuples.

    First, for fixed $j$ such that $1 \leq j \leq s$, we will find the number of $d$-tuples whose $j$th column is $\x$.
    In other words, we seek the number of $d$-tuples such that the $j$th entry in row $i$ is $x_i$, for each $i = 1, \ldots, d$.
    For each $i$, this implies that the $j-1$ entries to the left of $x_i$ lie in the set $[x_i]$, and that the $s-j$ entries to the right of $x_i$ lie in the set $\{x_i, x_i + 1, \ldots, n_i\}$, which contains $n_i + 1 - x_i$ elements.
    Hence by~\eqref{composition formula}, the number of ways to fill each row $i$ such that the $j$th entry is $x_i$ equals
    \[
    \binom{(j-1) + x_i - 1}{j-1} \binom{(s-j) + (n_i + 1 - x_i) - 1}{s-j},
    \]
    which simplifies to the expression in the theorem.
    Taking the product over all rows $i = 1, \ldots, d$, we obtain the number of $d$-tuples whose $j$th column equals $\x$, as desired.
    
    Finally, to obtain the number of times $\x$ occurs as any column in a $d$-tuple, we sum over all columns $j = 1, \ldots, s$.
\end{proof}

\section{Preliminaries: Stanley--Reisner theory}
\label{sec:ASC}

\subsection{Abstract simpicial complexes}
 \label{sub:simplicial complexes}

 Our exposition below is standard, and details can be found in~\cite{StanleyAC}*{Ch.~12}.
 Given a finite set $V$, an \emph{(abstract) simplicial complex} on $V$ is a collection $\Delta$ of subsets of $V$ such that
 \begin{itemize}
     \item $\{v\} \in \Delta$ for all $v \in V$;
     \item if $S \in \Delta$ and $R \subseteq S$, then $R \in \Delta$.
 \end{itemize}
 Elements of $V$ are called \emph{vertices}, and $V$ is called the \emph{vertex set}.  The elements of $\Delta$ are called \emph{faces}, and the \emph{dimension} of a face is one less than its cardinality.  The dimension of $\Delta$ is the maximum of the dimensions of its faces.  
 
Let $\Delta$ be a nonempty simplicial complex of dimension $a-1$. 
We denote by $f_i$ the number of faces of dimension $i$.  
In particular, $f_{-1} = 1$ since $\emptyset \in \Delta$.  
The \emph{$f$-vector} is the sequence $(f_0, \ldots, f_{a-1})$.
It will be more convenient for us to work instead with the \emph{$h$-vector} $(h_0, \ldots, h_a)$, where the numbers $h_k$ are defined as follows:
\[
h_\Delta(t) \coloneqq \sum_{k=0}^a h_k t^k = \sum_{i=0}^a f_{i-1} t^i (1-t)^{a-i}.
\]
We call $h_\Delta(t)$ the \emph{$h$-polynomial} of $\Delta$.
 
A maximal face in $\Delta$ (with respect to inclusion) is called a \emph{facet}. 
We say that $\Delta$ is \emph{pure} if every facet has the same dimension.  
A pure simplicial complex is said to be \emph{shellable} if there exists an ordering $F_1, \ldots, F_{r}$ of its facets with the following property: for all $i = 1, \ldots, r$, the power set of $F_i$ has a unique minimal element ${\rm R}(F_i)$ not belonging to the subcomplex generated by $F_1, \ldots, F_{i-1}$.  
Such an ordering is called a \emph{shelling}, and ${\rm R}(F_i)$ is called the \emph{restriction} of the facet $F_i$.  
Crucial to our method is the following combinatorial description of the $h$-polynomial: if $F_1, \ldots, F_r$ is a shelling of $\Delta$, then we have
\begin{equation}
    \label{h-vec as restrictions}
    h_\Delta(t) = \sum_{i=1}^r t^{\#{\rm R}(F_i)}.
\end{equation}
In other words, $h_k$ counts the number of facets whose restrictions have size $k$:
\[
    h_k = \#\{i : \#{\rm R}(F_i) = k\}.
\]

\subsection{Lexicographic shellings of posets}

Let $\Pi$ be a finite bounded poset, with $\lessdot$ denoting the covering relation.
Let $\mathcal{E} \coloneqq \{(a,b) : a \lessdot b\}$ be the set of edges of the Hasse diagram of $\Pi$.
A \emph{labeling} of $\Pi$ is a function $\la : \mathcal{E} \rightarrow \mathbb{Z}_{>0}$ assigning a positive integer to each edge of the Hasse diagram.
Each labeling $\la$ induces a lexicographic ordering on the set of saturated chains $a_1 \lessdot a_2 \lessdot \cdots \lessdot a_\ell$, via the lexicographic order on the \emph{label sequences} $(\la(a_1,a_2), \ldots, \la(a_{\ell-1},a_\ell))$.
Following Bj\"orner and Wachs~\cite{BW83}, we define a special kind of labeling known as an edge-lexicographical (EL) labeling:

\begin{definition}[\cite{BW83}*{Def.~2.1}]
We say that $\la$ is an \emph{EL-labeling} of $\Pi$ if, for all $a < c$ in $\Pi$,
there exists a unique saturated chain $a \lessdot b_1 \lessdot \cdots \lessdot b_\ell \lessdot c$ such that 
\begin{equation}
    \label{ascending chain}
    \la(a,b_1) \leq \la(b_1, b_2) \leq \cdots \leq \la(b_\ell, c),
\end{equation}
and this chain lexicographically precedes all other saturated chains $a \lessdot \cdots \lessdot c$.
\end{definition}

A chain with the property~\eqref{ascending chain} is called an \emph{ascending chain} with respect to $\la$.

The \emph{order complex} $\Delta(\Pi)$ is the simplicial complex whose faces are the chains in $\Pi$; hence the facets of $\Delta(\Pi)$ are the maximal chains in $\Pi$.
An EL-labeling of $\Pi$ induces a shelling order on the facets of $\Delta(\Pi)$, via the lexicographic order on the maximal chains \cite{Bjorner80}*{Thm.~2.3}.
Note that an EL-labeling does not guarantee that the maximal chains are totally ordered; nevertheless, arbitrarily breaking ties results in a shelling order.

Let $\la$ be an EL-labeling of $\Pi$, and $F$ a facet of $\Delta(\Pi)$.
An element $b \in F$ is said to be a \emph{descent} of $F$ with respect to $\la$, if $F$ contains $a \lessdot b \lessdot c$ such that $\la(a,b) > \la(b,c)$.
With respect to any shelling induced by $\la$, the restriction of each facet is precisely its set of descents:
\begin{equation}
    \label{R equals descents}
    {\rm R}(F) = \{ b : \text{$b$ is a descent of $F$}\}.
\end{equation}
(See~\cite{BW96}*{Thm.~5.8}.)
We will also use the term \emph{descent} in the context of label sequences, in the obvious sense: 
namely, $i$ is a descent of the label sequence $(\la_1, \ldots, \la_\ell)$ if $\la_i > \la_{i+1}$.
In this way, the number of descents in a facet equals the number of descents in its label sequence.

\subsection{The Stanley--Reisner ring}

Let $\Delta$ be a simplicial complex on the vertex set $V$.
Let $K$ be a field, and consider the polynomial ring $K[V] \coloneqq K[z_v : v \in V]$, where we regard the elements of $V$ as indeterminates.
Given a subset $U \subseteq V$, we will use the shorthand
\[
z_U \coloneqq \prod_{v \in U} z_v \qquad \text{and} \qquad K[U] \coloneqq K[v : v \in U ].
\]

Let $I_\Delta$ be the ideal of $K[V]$ generated by all monomials $z_U$ such that $U \not\in \Delta$.
Such a $U$ is called a \emph{nonface} of $\Delta$, and it is easy to see that $I_\Delta$ is actually generated by those nonfaces which are minimal (i.e., which contain no proper nonface).
The \emph{support} of a monomial $m \in K[V]$ is the set $\{ v \in V : \text{$z_v$ divides $m$}\}$.
A $K$-basis for $I_\Delta$ is given by the monomials whose support is not contained in $\Delta$.   

The quotient $K[\Delta] \coloneqq K[V]/I_\Delta$ is called the \emph{Stanley--Reisner} ring of $\Delta$.  
It is clear that $K[\Delta]$ has a $K$-basis consisting of the monomials whose support is a face of $\Delta$ (where we identify these monomials with their images in the quotient ring).
% We then have a bijection
% \begin{align*}
%     \T(d,n) &\longleftrightarrow \text{$K$-basis of $K[\Delta_n]_d$},\\
%     T &\longleftrightarrow \prod_{i,j} x_{ij}^{T_{ij}},
% \end{align*}
% obtained by viewing each matrix in $\T(d,n)$ as the exponent matrix of a monomial in $K[\Delta_n]_d$.
% It follows that 
%     \begin{equation}
%         \label{prod monomials}
%         \prod_{\mathbf{m} \in K[\Delta_n]_d} \hspace{-10pt}\mathbf{m} = \prod_{(i,j) \in \Pi_n} x_{ij}^{\SS(d,n)_{ij}},
%     \end{equation}
% i.e., $\SS(d,n)$ is the exponent matrix of the product of all monomials in $K[\Delta_n]_d$.
% In other words, our main problem of finding $\SS(d,n)$ can be reinterpreted as writing down a formula for this product of monomials.

Each shelling of $\Delta$ induces a \emph{Stanley decomposition} of the Stanley--Reisner ring:
\begin{equation}
    \label{Stanley decomp general}
    K[\Delta] = \bigoplus_{F} K[F] \: z_{{\rm R}(F)},
\end{equation}
where the direct sum ranges over the facets $F$, and their restrictions ${\rm R}(F)$ are determined by the shelling.
Crucially, each monomial in $K[\Delta]$ lies in exactly one summand of~\eqref{Stanley decomp general}.
Since $I_\Delta$ is generated by homogeneous polynomials (in fact, by monomials), the quotient $K[\Delta]$ inherits from $K[V]$ the natural grading by degree.
Writing $K[\Delta]_{s}$ to denote the graded component consisting of homogeneous polynomials of degree $s$, we can restrict~\eqref{Stanley decomp general} to a decomposition of each component:
\begin{equation}
    \label{Stanley decomp graded}
    K[\Delta]_s = \bigoplus_k \bigoplus_{\substack{F:\\ \#{\rm R}(F) = k}} K[F]_{s-k} \: z_{{\rm R}(F)},
\end{equation}
where $k$ ranges from 0 to the size of the largest restriction ${\rm R}(F)$.

\subsection{Application to the problem}

In this final subsection, we apply the general theory above to the poset $\Pi_{\x}$ defined in~\eqref{def:Pi}.
We write $\Delta_\x \coloneqq \Delta(\Pi_\x)$ for its order complex.
The facets of $\Delta_\x$ are the maximal chains $\mathbf{1} \lessdot \cdots \lessdot \x$.
Clearly for any facet $F$ of $\Delta_\x$, we have
\begin{equation}
    \label{size of F}
    \#F = |\x| - d  + 1,
\end{equation}
 so $\Delta_\x$ is indeed pure.

Let $\mathbf e_i$ denote the vector whose $i$th coordinate is 1, with 0's elsewhere.
If $\mathbf a \lessdot \mathbf b$, then clearly we have $\mathbf b = \mathbf a + \mathbf e_i$ for some $1 \leq i \leq d$.
We define the following labeling on $\Pi_\x$:
\begin{equation}
    \label{lambda}
    \la(\mathbf a, \mathbf b) = i \iff 
    \mathbf b = \mathbf a + \mathbf e_i.
\end{equation}
For example, if $\mathbf a = (3,6,4,1)$ and $\mathbf b = (3,7,4,1)$, then $\la(\mathbf a, \mathbf b) = 2$.
See Figure~\ref{fig:facet} for a visualization in the case where $\x = (3,3,3)$.

\begin{figure}[t]
    \centering
    \tdplotsetmaincoords{70}{110} % 

\tdplotsetmaincoords{70}{120} % set viewpoint 
\tdplotsetrotatedcoords{0}{0}{0} %<- rotate around (z,y,z)
\begin{tikzpicture}[scale=2,tdplot_rotated_coords,
            blackBall/.style={ball color = black!80},
                    borderBall/.style={ball color = white,opacity=.25}, 
                    very thick, lightgray]

\foreach \x in {0,1,2}
   \foreach \y in {0,1,2}
      \foreach \z in {0,1,2}{
           %#####################################################
           \ifthenelse{  \lengthtest{\x pt < 2pt}  }{
             \draw (\x,\y,\z) -- (\x+1,\y,\z);
             \shade[blackBall] (\x,\y,\z) circle (0.025cm); 
           }{}
           %#####################################################
           \ifthenelse{  \lengthtest{\y pt < 2pt}  }{
               \draw (\x,\y,\z) -- (\x,\y+1,\z);
               \shade[blackBall] (\x,\y,\z) circle (0.025cm);
           }{}
           %#####################################################
           \ifthenelse{  \lengthtest{\z pt < 2pt}  }{
               \draw (\x,\y,\z) -- (\x,\y,\z+1);
               }{}
}

\draw[ultra thick,black] (2,0,2) -- node[above] {3} (1,0,2) -- node[above] {2} (1,1,2) -- node[below right] {1} (1,1,1) -- node[above right] {2} (1,2,1) -- node[above right] {1} (1,2,0) -- node[below] {3} (0,2,0);
\draw[very thick, lightgray] (2,1,2) -- (2,2,2) -- (2,2,1);

\shade[blackBall] (1,0,2) circle (0.08cm); 
\shade[blackBall] (1,1,2) circle (0.08cm);
\shade[blackBall] (1,2,1) circle (0.08cm);

\shade[borderBall] (1,0,2) circle (0.15cm);
\shade[borderBall] (1,1,2) circle (0.15cm);
\shade[borderBall] (1,2,1) circle (0.15cm);

\shade[blackBall] (2,2,2) circle (0.025cm);
\shade[blackBall] (2,2,1) circle (0.025cm);
\shade[blackBall] (2,2,0) circle (0.025cm);

\fill[black] 
(2,0,2) circle (0pt) node[left] {$\mathbf{1} = (1,1,1)$}
(0,2,0) circle (0pt) node[right] {$\x = (3,3,3)$};

\end{tikzpicture}
    \caption{Visualization of a facet $F$ in the order complex $\Delta_\x$, where $\x = (3,3,3)$.
    Starting in the upper-left at $\mathbf{1}$, we imagine the coordinate vector $\mathbf{e}_1$ pointing downward, $\mathbf{e}_2$ pointing to the right, and $\mathbf{e}_3$ pointing away from the viewer.
    With respect to the labeling $\la$ in~\eqref{lambda}, the label sequence $(3,2,1,2,1,3)$ of $F$ is shown in the figure.
    The descents are indicated by the three large dots; by~\eqref{R equals descents}, these are the elements of ${\rm R}(F)$.}
    \label{fig:facet}
\end{figure}
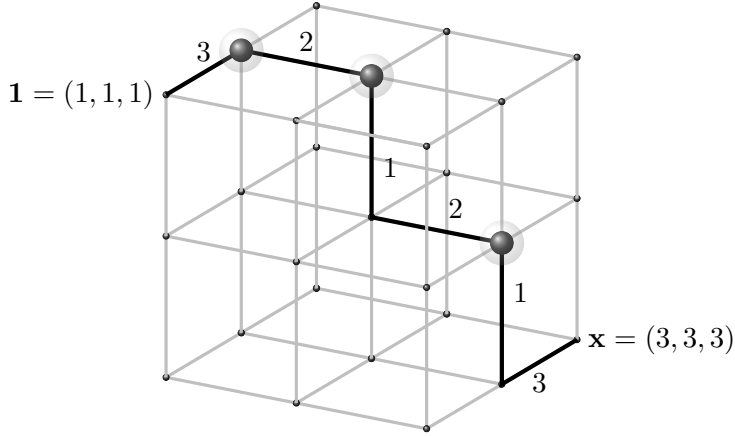

It is easy to see that $\la$ is an EL-labeling: for $\mathbf a < \mathbf b$, the unique ascending chain $\mathbf a \lessdot \cdots \lessdot \mathbf b$ with respect to $\la$ is obtained from $\mathbf a$ by first adding $\mathbf e_1$ a total of $b_1 - a_1$ times, then adding $\mathbf e_2$ a total of $b_2 - a_2$ times, etc., and finally adding $\mathbf e_d$ a total of $b_d - a_d$ times.
Moreover, this chain precedes any other maximal chain between $\mathbf a$ and $\mathbf b$.
Being an EL-labeling, $\la$ induces a unique shelling of $\Delta_\x$, since the lexicographical order (in this case) gives a total ordering of the facets.

We now turn to our main problem: writing down a formula for each component of $\Sigma^{1}_{\n,s}$, which we recall from~\eqref{Sigma1ns def} is the sum of all tensors in $\T^{1}_{\n,s}$.
To this end, note that a $K$-basis for $K[\Delta_\n]$ is given by the monomials whose support is a chain in $\Pi_\n$.
Restricting to the degree-$s$ component, we observe the bijection
\begin{align*}
\T^{1}_{\n,s} & \longleftrightarrow \text{$K$-basis of $K[\Delta_\n]_s$},\\
T & \longleftrightarrow \prod_{\x \in \Pi_\n} z_\x^{T[\x]}.
\end{align*}
Under this correspondence, adding tensors corresponds to multiplying monomials.
Therefore, letting $m$ range over the monomials, we have
\begin{equation}
    \label{adding arrays is mult monomials}
    \prod_{\mathclap{m \in K[\Delta_\n]_s}} m = \prod_{\x \in \Pi_\n} z_{\x}^{\Sigma^{1}_{\n,s}[\x]}.
\end{equation}
Hence our main problem is equivalent to finding the exponent of each indeterminate $z_\x$ in the product of monomials on the left-hand side of~\eqref{adding arrays is mult monomials}.
We will do this in Section~\ref{sec:result version 2}.  
Before that, however, we must explain and exploit the fact that the $h$-polynomial of $\Delta_\x$ is actually a well-known object called the \emph{multiset Eulerian polynomial}.

\section{Multiset Eulerian polynomials}
\label{sec:Eulerian polys}

Let $\mathbf{p} = (p_1, \ldots, p_d)$.  A \emph{multipermutation} of the multiset $\{1^{p_1}, 2^{p_2}, \ldots, d^{p_d}\}$ is a word $\pi = \pi_1 \cdots \pi_{|\mathbf{p}|}$ in which $i$ appears exactly $p_i$ times, for each $1 \leq i \leq d$.
Let $\mathfrak{S}_{\mathbf p}$ be the set of all such multipermutations.
A \emph{descent} of a multipermutation $\pi$ is an index $i$ such that $\pi_i > \pi_{i+1}$.
Let ${\rm des}(\pi)$ denote the number of descents of $\pi$.
Then the \emph{multiset Eulerian polynomial} $A_{\mathbf p}(t)$ is defined to be
\[
    A_{\mathbf p}(t) \coloneqq \sum_{\pi \in \mathfrak{S}_{\mathbf p}} t^{\:{\rm des}(\pi)}.
\]
(The special case $A_{\mathbf 1}(t)$ is just the $d$th Eulerian polynomial $A_d(t)$, i.e., the descent-generating function over the symmetric group $\mathfrak{S}_d$.  
See, for example,~\cite{Stanley1}*{p.~22}, although the convention there is to multiply through by $t$.)
The multiset Eulerian polynomial occurs as the numerator of the following generating function, due to MacMahon~\cite{MacMahon}*{p.~211}:
\begin{equation}
    \label{MacMahon gf}
    \frac{A_{\mathbf p}(t)}{(1-t)^{|\mathbf p|+1}} = \sum_{\ell = 0}^{\infty} \prod_{i=1}^d\binom{p_i + \ell}{\ell} t^\ell.
\end{equation}
From~\eqref{MacMahon gf} we can obtain an explicit expression for the coefficients in $A_{\mathbf p}(t)$; see also~\cites{DR69,Abramson75} for combinatorial proofs of the formula below.
(These coefficients are known as ``Simon Newcomb numbers.'')
Writing $[t^k]$ for the coefficient of $t^k$, we have
\begin{equation}
    \label{formula Euler numbers}
    [t^k] A_{\mathbf p}(t) = \#\Big\{ \pi \in \mathfrak{S}_{\mathbf p} : {\rm des}(\pi) = k\Big\} = 
    \sum_{\ell = 0}^k (-1)^\ell \binom{|\mathbf{p}| + 1}{\ell} \prod_{i=1}^d \binom{p_i + k - \ell}{k-\ell}.
\end{equation}
It is shown in~\cite{DR69}*{Lemma~2} that the maximum number of descents, i.e., the degree of $A_{\mathbf p}(t)$, equals 
\begin{equation}
    \label{degree of AMt}
    |\mathbf{p}| - \max(\mathbf p).
\end{equation}

\begin{lemma}
    \label{lemma: h-poly}
    The $h$-polynomial $h_{\x}(t)$ of $\Delta_\x$ is the multiset Eulerian polynomial $A_{\x-\mathbf{1}}(t)$.
% \[
%     h_\x(t) = \sum_{k=0}^{|\x| - \max(\x) - d  +1} \left[\sum_{\ell = 0}^k (-1)^\ell \binom{|\x| - d + 1}{\ell} \prod_{i=1}^d \binom{k - \ell + x_i - 1}{k - \ell}\right]t^k. 
% \]
\end{lemma}

\begin{proof}
Recall the labeling $\la$ in~\eqref{lambda}, which induces a shelling of $\Delta_\x$.
With respect to $\la$, the label sequence of each facet $\mathbf{1} \lessdot \cdots \lessdot \x$ contains $x_i - 1$ copies of the label $i$, for each $i = 1, \ldots, d$; conversely, each possible permutation of these labels (where the copies of each $i$ are indistinguishable from each other) is the label sequence of a unique facet.
Hence we have a bijection between the set of facets of $\Delta_\x$ and the set $\mathfrak{S}_{\x - \mathbf{1}}$.
By~\eqref{R equals descents}, the size of ${\rm R}(F)$ equals the number of descents in $F$, which in turn equals the number of descents in the label sequence of $F$.
Therefore, comparing~\eqref{h-vec as restrictions} and~\eqref{formula Euler numbers}, it is clear that $h_\x(t) = A_{\x-\mathbf{1}}(t)$.  
\end{proof}

%\question{Need to read and better understand the paper ``On the $\gamma$-positivity of multiset Eulerian polynomials,'' by Lin, Xu, Zhao.  They study our polynomials $h_{\x}$ where $\x$ is constant, i.e., where our array is a hypercube.}

%\question{Mention that $h_\x(t)$ is essentially the ``hit polynomial'' for the Young diagram associated with $\x$; see ``Q-Hit Polynomials Have Only Real Roots,'' by Li-Ping Mo, specifically Proposition 1.11.}  

\begin{rem}
    Combining~\eqref{h-vec as restrictions} with the Stanley decomposition~\eqref{Stanley decomp general}, it is easy to see that the Hilbert series of a Stanley--Reisner ring $K[\Delta]$ is given by
    \[
        \frac{h_\Delta(t)}{(1-t)^{\#F}},
    \]
    where $F$ is any facet of $\Delta$ (since $\Delta$ is assumed to be pure).
    Thus by Lemma~\eqref{lemma: h-poly} and~\eqref{size of F}, and by MacMahon's expansion~\eqref{MacMahon gf}, our particular ring $K[\Delta_\n]$ has the Hilbert series
    \[
        \frac{A_{\n - \mathbf{1}}(t)}{(1-t)^{|\n| - d + 1}} = \sum_{\ell = 0}^\infty \prod_{i=1}^d \binom{n_i + \ell - 1}{\ell} t^\ell.
    \]
    This is also the Hilbert series of the coordinate ring of the set of simple (also called pure, or decomposable) tensors in $\T_\n$, which is isomorphic to $K[\Delta_\n]$.
    See~\cite{Morales}*{Thm.~5}, where this ring is also viewed as the toric ring defined by the Segre embedding of $\mathbb{P}^{n_1} \times \cdots \times \mathbb{P}^{n_d}$.
\end{rem}

Combining~\eqref{degree of AMt} with Lemma~\ref{lemma: h-poly}, we note that the degree of $h_\x(t) = A_{\x-\mathbf{1}}(t)$ equals
\begin{equation}
    \label{degree hx}
    |\x - \mathbf{1}| - \max(\x-\mathbf{1}) = |\x| - \max(\x) - d + 1.
\end{equation}
Equivalently,~\eqref{degree hx} is the maximum size of ${\rm R}(F)$, taken over all facets $F$ of $\Delta_{\x}$.

\section{Main result, second version}
\label{sec:result version 2}

\begin{theorem}
    \label{thm:main result h polys}
    For each $\x \in \Pi_\n$, we have
    \[
    \Sigma^{1}_{\n,s}[\x] = \sum_{k=0}^{\min\{\omega(\n,\x), \: s-1\}} \binom{|\n| - d + s - k}{s-k-1} \cdot [t^k] \, A_{\x-\mathbf 1}(t) A_{\n - \x}(t),
    \]
    where $\omega(\n,\x) \coloneqq |\n| - \max(\x) - \max(\n-\x) - d + 1$.
\end{theorem}

Before giving the proof, we record two easy counting lemmas.

\begin{lemma}
    \label{lemma:alpha}
    Let $\x \in \Pi_n$, and let $F \ni \x$ be a facet of $\Delta_\n$.
    Then $\binom{|\n| - d + s - k}{s-k-1}$ equals the exponent of $z_\x$ in the product of all monomials in
\begin{equation}
    \label{summand in lemma}
    K[F]_{s-k-1} z_{_{{\rm R}(F) \cup \{\x\}}}.
\end{equation}
\end{lemma}

\begin{proof}
It suffices to show that
\[
\binom{|\n| - d + s - k}{s-k-1} = \Big(\text{\# monomials in \eqref{summand in lemma}}\Big) \Big(\text{average exponent of $z_\x$ in each monomial}\Big).
\]
The number of monomials in~\eqref{summand in lemma} equals the number of monomials in $K[F]_{s-k-1}$.
This number, in turn, equals the number of weak compositions of the degree $s-k-1$ into $\#F$ many parts.
Thus, recalling from~\eqref{size of F} that $\#F = |\n| - d  +1$ for any facet $F$ of $\Delta_\n$, and using the elementary formula~\eqref{composition formula}, we have
\begin{equation}
    \label{number monomials}
    \text{\# monomials in~\eqref{summand in lemma}} = \binom{s-k-1 + |\n| - d}{|\n| - d}.
\end{equation}
The average exponent of $z_\x$, taken over all the monomials in $K[F]_{s-k-1}$, equals the degree $s-k-1$ divided by the number of variables $\#F$.
Adding 1 to this average to account for the factor of $z_\x$ present in $z_{_{{\rm R}(F) \cup \{\x\}}}$ in~\eqref{summand in lemma}, we obtain 
\begin{equation}
    \label{avg exp}
    \text{average exponent of $z_\x$ in each monomial} = 1 + \frac{s-k-1}{|\n| -d + 1} = \frac{|\n| - d + s - k}{|\n| - d +1}.
\end{equation}
Multiplying the expressions in~\eqref{number monomials} and~\eqref{avg exp}, we obtain
\[
\binom{s-k-1 + |\n| - d}{|\n| - d} \cdot \frac{|\n| - d + s - k}{|\n| - d +1} = \binom{|\n| - d + s - k}{|\n| - d + 1} = \binom{|\n| - d + s - k}{s-k-1}. \qedhere
\]
\end{proof}

\begin{lemma}
    \label{lemma: coeff tk}
    For $\x \in \Pi_\n$, the coefficient
    \[
    [t^k] \, A_{\x-\mathbf{1}}(t) A_{\n - \x}(t)
    \]
    equals the number of facets $F \ni \x$ of $\Delta_\n$, such that $\#\big({\rm R}(F) \setminus \{\x\}\big) = k$.
\end{lemma}

\begin{proof}
    Every facet $F \ni \x$ of $\Delta_\n$ can be written uniquely as the union of two saturated chains
    \[
    F' : \mathbf 1 \lessdot \cdots \lessdot \x \qquad \text{and} \qquad F'' : \x \lessdot \cdots \lessdot \n,
    \]
    which intersect only at $\x$.
    Clearly $F'$ can be any facet of $\Delta_\x$, viewed as a subposet of $\Delta_\n$.
    Likewise, $F''$ can be any facet of $\Delta_{\n + \x - \mathbf 1}$, viewed as a subposet of $\Delta_\n$ after translating coordinates by $\x-\mathbf 1$.
    Therefore $\#({\rm R}(F) \setminus \{\x\})  = {\rm R}(F') + {\rm R}(F'')$.
    Thus by~\eqref{h-vec as restrictions}, we have  \begin{align*} 
    h_\x(t) h_{\n - \x + \mathbf 1}(t) &= \left(\sum_{F'} t^{\#{\rm R}(F')} \right)\left(\sum_{F''} t^{\#{\rm R}(F'')}\right)\\
    &= \sum_{F', F''} t^{\#{\rm R}(F') + \#{\rm R}(F'')}\\
    &= \sum_{F \ni \x} t^{\#({\rm R}(F) \setminus \{\x\})},
    \end{align*}
    where the sums range over facets $F$, $F'$, and $F''$ of $\Delta_\n$, $\Delta_\x$, and $\Delta_{\n - \x + \mathbf 1}$, respectively.
    By Lemma~\ref{lemma: h-poly}, we can rewrite $h_{\x}(t) h_{\n -\x + \mathbf 1}(t)$ as $A_{\x-\mathbf{1}}(t) A_{\n - \x}(t)$.
\end{proof}

\begin{proof}[Proof of Theorem \ref{thm:main result h polys}]

By~\eqref{adding arrays is mult monomials}, we know that $\Sigma^{1}_{\n,s}[\x]$ equals the exponent of $z_\x$ in the product of all monomials in the graded component $K[\Delta_{\n}]_s$, which by~\eqref{Stanley decomp graded} has the decomposition
\begin{equation}
    \label{Stanley decomp in proof first}
    K[\Delta_\n]_s = \bigoplus_k \bigoplus_{\substack{F:\\ \#{\rm R}(F) = k}} K[F]_{s-k} \: z_{{\rm R}(F)},
\end{equation}
where the $F$'s in the inside sum are facets of $\Delta_\n$.
The outside sum in~\eqref{Stanley decomp in proof first} ranges from $k=0$ to $k= \min\{|\n| - \max(\n) - d + 1, \: s\}$; this follows from \eqref{degree hx} and from the fact that the degree $s-k$ must be nonnegative.
Obviously, the only monomials contributing to the exponent of $z_\x$ are those divisible by $z_\x$;
hence we may ignore all summands in~\eqref{Stanley decomp in proof first} such that $\x \not\in F$.
If $\x \in F$, then the subspace of $K[F]_{s-k}$ spanned by the monomials divisible by $z_\x$ is
\[
K[F]_{s-k-1} z_\x.
\]
Then since $\x$ may or may not lie in ${\rm R}(F)$, the subspace of $K[F]_{s-k} z_{{\rm R}(F)}$ spanned by monomials divisible by $z_\x$ is
\[
K[F]_{s-k-1} z_\x z_{_{{\rm R}(F) \setminus \{\x\}}} = K[F]_{s-k-1} z_{_{{\rm R}(F) \cup \{\x\}}}.
\]
Combining this with~\eqref{Stanley decomp in proof first}, we conclude that $\Sigma^{1}_{\n,s}[\x]$ equals the exponent of $z_\x$ in the product of all monomials in
\begin{equation}
    \label{Stanley decomp in proof second}
    \bigoplus_k \bigoplus_{\substack{F \ni \x:\\ \#({\rm R}(F) \setminus \{\x\}) = k}} K[F]_{s-k-1} z_{_{{\rm R}(F) \cup \{\x\}}}.
\end{equation}
Now applying Lemma~\ref{lemma:alpha} and Lemma~\ref{lemma: coeff tk} to~\eqref{Stanley decomp in proof second}, we see that the desired exponent of $z_\x$ equals
\begin{align*}
    \Sigma^{1}_{\n,s}[\x] &= \sum_k \sum_{\substack{F \ni \x:\\ \#({\rm R}(F) \setminus \{\x\}) = k}} \binom{|\n| - d + s - k}{s-k-1}\\[2ex]
    &= \sum_k \binom{|\n| - d + s - k}{s-k-1} \cdot [t^k] \, A_{\x - \mathbf{1}} (t) A_{\n - \x} (t),
\end{align*}
where the nonzero summands are those for which $k$ is less than or equal to both $s-1$ (otherwise the binomial coefficient is zero) and the degree of $A_{\x-\mathbf{1}} (t) A_{\n - \x}(t)$.
This degree is easily computed to be $\omega(\n,\x)$, using \eqref{degree hx}.
\end{proof}

\bibliographystyle{plain}
\bibliography{references}

\end{document}